\documentclass{amsart}

\usepackage{amssymb}
\usepackage[stretch=10,shrink=10]{microtype}
\usepackage[showframe=false, margin = 1.5 in]{geometry}
\usepackage{mathtools}
\usepackage[shortlabels]{enumitem}
\usepackage{theoremref}
\usepackage{tikz}
\usepackage{subfigure, float, color}
\usepackage{comment}
\usepackage{tikz-qtree}

\newcommand{\HOX}[1]{\marginpar{\footnotesize #1}}
\renewcommand{\emptyset}{\varnothing}

\newcommand{\f}{\frac}
\newcommand{\rootvertex}{\varnothing}
\newcommand{\ind}[1]{\mathbf{1}{\{ #1 \}}}
\newcommand{\abs}[1]{|#1 |}

\newcommand{\wt}{\widetilde}

\newcommand{\Iu}{I^\uparrow}
\newcommand{\Id}{I^\downarrow}
\newcommand{\Ju}{J^\uparrow}
\newcommand{\Jd}{J^\downarrow}

\DeclareMathOperator{\Ber}{Ber}
\DeclareMathOperator{\Bin}{Bin}

\DeclareMathOperator{\RFM}{RFM}
\DeclareMathOperator{\SFM}{SFM}
\DeclareMathOperator{\FM}{FM}

\newtheorem{thm}{Theorem}
\newtheorem{lemma}[thm]{Lemma}
\newtheorem{prop}[thm]{Proposition}

\theoremstyle{remark}
\newtheorem{remark}[thm]{Remark}
\theoremstyle{definition}

\title[]{The frog model on trees with drift}
\author[]{Erin Beckman}

\address{Department of Mathematics, Duke University}
\email{ebeckman@duke.edu}
\author[]{Natalie Frank}
\address{Department of Mathematics, New York University}
\email{nf1066@nyu.edu}

\author[]{Yufeng Jiang}
\address{Department of Mathematics, Duke University}
\email{yufeng.jiang@duke.edu}

\author[]{Matthew Junge}
\address{Department of Mathematics, Duke University}
\email{jungem@math.duke.edu}

\author[]{Si Tang}
\address{Department of Mathematics, Duke University}
\email{si.tang@duke.edu}

\begin{document}

\maketitle

\begin{abstract}
	We provide a uniform upper bound on the minimal drift so that the one-per-site frog model on a $d$-ary tree is recurrent. 
To do this, we introduce a subprocess that couples across trees with different degrees. Finding couplings for frog models on nested sequences of graphs is known to be difficult. The upper bound comes from combining the coupling with a new, simpler proof that the frog model on a binary tree is recurrent when the drift is sufficiently strong. Additionally, we describe a coupling between frog models on trees for which the degree of the smaller tree divides that of the larger one. This implies that the critical drift has a limit as $d$ tends to infinity along certain subsequences.
%This leads to a new, simpler technique to prove recurrence when the drift is sufficiently strong.
	%The proof makes use of a novel subprocess that we call the recursive frog model. With the drift set constant, the recursive frog model can be coupled on different degree trees. It also leads to a new proof of recurrence for l This begins to addresses the question of how the frog model behaves as the graph is augmented
\end{abstract}

\section{Introduction}

We study the one-per-site \emph{frog model with drift} on the rooted $d$-ary tree $\mathbb T_d$.  Initially there is a single awake frog at the root and one sleeping frog at each non-root vertex. Awake frogs move towards the root with probability $p$, and otherwise move away from the root to a uniformly sampled child vertex. Frogs at the root always move to a uniformly sampled child vertex. Whenever an awake frog visits a site with a sleeping frog, the sleeping frog wakes up and begins its own independent $p$-biased random walk. Denote this process by $\FM(d,p)$ and the total number of visits to the root by $V(d,p)$.
The process is \emph{recurrent} if $V(d,p)$ is infinite almost surely, and is otherwise \emph{transient}.
The almost sure requirement is not overly stringent, because the probabilty the root is visited infinitely often in $\FM(d,p)$ satisfies a 0-1 law by Kosygnia and Zerner's general result \cite[Theorem 1]{kosygina01}.

There is a history of investigating recurrence for the frog model with drift. It was first studied by Gantert and Schmidt with i.i.d\ $\eta$ frogs per site and a drift in the $e_1$ direction on $\mathbb Z$ \cite{nina_drift1}. They showed that the process is recurrent if and only if $E \log \eta = \infty$ regardless of the drift.  A followup work by Ghosh et.\ al.\ studied the range of the frog model in the transient case \cite{ghosh_drift}. Similar observations were made by  Rosenberg when the frog paths are Brownian motions in $\mathbb R$ \cite{josh_drift, josh_drift2}. The question is more subtle and challenging in higher dimensions.  D\"obler and Pfeifroth showed that the frog model is recurrent on $\mathbb Z^d$ for $d\geq 2$ so long as $E \log^{(d+1)/2} \eta = \infty$ \cite{doblerdrift}. 
% Kosygnia and Zerner proved an analogous statement \HOX{what was different about their result? -EB} \cite{kosygina01}. 
 It was open for some time whether, unlike the $d=1$ case, there is a phase transition as the drift is varied. This was recently answered by D\"obler et.\ al.\ \cite{nina_drift2}. With one sleeping frog at each site of $\mathbb Z^d$, they found that recurrence depends on the strength of the drift with notably different behavior in $d=2$ and $d\geq 3$.
 %The argument uses a clever relationship to site percolation.
%
We prove here that transience and recurrence of $\FM(d,p)$ also depends on the drift.
%This is less surprising than what was discovered on the lattice, because it is known that the frog model on $d$-ary trees with frogs following simple random walk paths undergoes a phase transition as either the degree of the tree, or the initial density of frogs is varied \cite{HJJ1,HJJ2}.
%More precisely, for the frog model with $\Poi(\mu)$ frogs per site there exists $\mu_d$ such that for $\mu< \mu_d$ the frog model on the $d$-ary tree is transient, and for $\mu> \mu_d$ it is recurrent.

Trees are a natural setting to study the frog model with drift, because the graph structure already induces one. Indeed, $\FM_d=\FM(d,1/(d+1))$ is the frog model with simple random walk paths.
Hoffman et.\ al.\ proved that $\FM_2$ is recurrent, but that $\FM_d$ is transient for $d \geq 5$ \cite{HJJ1}. What happens when $d=3$ and $d=4$ for the one-per-site frog model is not currently known. However, followup work by Hoffman et.\ al.\ showed that the frog model with unbiased random walks can be made recurrent for any $d$ so long as  $\Omega(d)$ sleeping frogs are placed at each site \cite{HJJ2, JJ3_log, JJ3_order}. So, there is a phase transition as we change the degree of the tree, or the initial density of asleep frogs.

%Knowing where the phase transition occurs gives us information about how the set of visited vertices grows. Letting $\rho = p/1-p$, the probability an awake frog at distance $n$ ever visits the root in $\FM(d,p)$ is $\rho^n$. Thus, if $\FM(d,p)$ transient then the number of frogs woken at level $n$ must be asymptotically smaller than $\rho^{-n}$.

Since $\FM_d$ is known to be transient for $d \geq 5$, it is natural to ask how much drift 
$$p_d = \inf \{ p \colon \FM(d,p) \text{ is recurrent}\}$$
is needed to make the process recurrent.
From this perspective, the main theorem of \cite{HJJ1} is that $p_2= 1/3$.
In general, we know that $p_d \leq 1/2$, because the initially awake frog will return to the root infinitely often when $p\geq 1/2$.
 A simple argument shows that if $p < 1/(d+1)$ then, even with all frogs initially awake, there are only finitely many expected visits to the root. This immediately gives the bounds $ 1 /({d+1}) \leq p_d \leq \f 12.$

It is not much more difficult to establish a non-vanishing lower bound on $p_d$.
The frog model is dominated by the branching random walk (BRW) on $\mathbb T_d$ in which particles do not branch when moving towards the root (with probability $p$), but split in two when moving away. This corresponds to $\FM(\infty,p)$. This BRW is a common tool for analyzing the frog model. By replacing $1/(d+1)$ with $p$ in the calculation at \cite[Proposition 15]{HJJ2} it follows that the BRW, and thus $\FM(d,p)$, is transient for $p< q^* =(2-\sqrt 2)/4 \approx .1464$. Thus, $p_d \geq q^*$. 
%And, by looking at the process after two steps, it is easy to see that the inequality is strict.
 Our main contribution is an upper bound.
\begin{thm}\thlabel{thm:main}
	$p_d \leq .4155$ for all $d \geq 2$.
\end{thm}

\begin{proof}
This follows from \thref{lem:recRFM} combined with \thref{prop:mono} and \thref{prop:2}.
%The lower bound is a consequence of \thref{lem:brw}.
\end{proof}

It is interesting to ask how the frog model relates to the dominating BRW. The extra drift the frog model needs to be recurrent, $p_d-q^*$, is one way to measure the difference. By using a BRW that approximates two steps of the frog model, it is not overly taxing to show that $p_d -q^* >0$ for all $d$ (see \cite[Proposition 19]{HJJ1} for an example of a more refined BRW). Since the dominating BRW corresponds in a sense to $d=\infty$, it is natural to ask if $p_d \to q^*$, and, if so, at what rate? The answer is not obvious, as sites near the root gradually get visited. Thus, as time elapses, the frog model branches less. The region with less branching may grow quickly. Hoffman et.\ al.\ proved in \cite{HJJ_shape} that, when the density of frogs is $\Omega(d^2)$, the set of activated sites on the $d$-ary tree contains a linearly expanding ball. Awake frogs in this region cause no branching. We are not sure if this prevents $p_d$ from converging to $q^*$. In fact, we are not sure whether $p_d$ converges at all. This question of convergence is the second reason we are interested in $\FM(d,p)$.

Coupling frog models on different graphs is known to be difficult. Past work  by Fontes et.\ al.\ established that the critical probability for the frog model with death is not monotonic in the graph \cite{mono}. However, Lebensztayn et.\ al.\ in \cite{improved} conjectured that monotonicity holds on regular trees. No coupling has ever been exhibited between the process on different degree trees. This would be nice because it might help understand how the frog model behaves on random trees and investigate the convergence of $p_d$. 
%We make some progress on coupling frog models on different degree trees.

We say that one frog model \emph{dominates} another,  denoted  $\FM(d,p) \preceq \FM(d',p')$, if there is a coupling so that every awake frog in $\FM(d,p)$ is coupled to an awake frog in $\FM(d',p')$ at the same distance from the root. Intuitively, this means that the set of awake frogs in $\FM(d,p)$ can be embedded in the set for $\FM(d',p')$. An immediate consequence of domination is that $V(d,p) \preceq V(d',p')$ in the usual sense of stochastic domination.

 It ought to hold that $\FM(d,p)\preceq \FM(d+1,p)$. This is because the drift is the same, but there are significantly more frogs in the higher degree tree. Despite considerable effort, we were unable to construct such a coupling. It remains an open problem to prove that $p_{d+1} \leq p_d$, and thus that $p_d$ has a limit. Additionally, there is no obvious coupling so that $\FM(d,p)$ visits the root less than $\FM(d,p')$ does when $p<p'$. Although preposterous, we cannot rule out the  possibility $\FM(d,p)$ switches between being transient and recurrent multiple times as we increase $p$. 

The obvious coupling to try between $\FM(d,p)$ and $\FM(d+1,p)$ is to have paired frogs mimic one anothers' displacement from the root, but to move to uniformly chosen vertices when moving away. One can readily find realizations where the frog on $\mathbb T_d$ wakes a new frog, while the coupled frog on $\mathbb T_{d+1}$ does not. This breaks the coupling. We tried several more sophisticated couplings with no luck.

One special case in which a coupling works is for trees in which the degree of the smaller  tree divides that of the larger tree.

\begin{prop} \thlabel{prop:coupling}
$\FM(d,p) \preceq \FM(kd,p)$ for all $k \geq 1$.
\end{prop}

\noindent The argument relies on a natural way to embed copies of $\mathbb T_{k}$ into  $\mathbb T_{kd}$. It does not appear to generalize to any other degrees. It follows that the subsequences $(p_{d^k})_{k=1}^\infty$ converge to  some $p^*(d)$ for each $d \geq 2$. So, there is a limit along certain subsequences, but we are not sure if $p^{*}(d) = q^*$.

%$\mathbb T_{k^d}$ embeds very naturally in $\mathbb T_{k^{d+1}}$ for $k \geq 2$ and any $d$. The coupling is described in \thref{prop:arithmetic}.
 We make more substantial progress coupling across different graphs with a subprocess of $\FM(d,p)$ that we call the \emph{recursive frog model} $\RFM(d,p)$. It is obtained by trimming and halting the random walk paths of awake frogs. This ensures that $\RFM(d,p)$ visits the root less than $\FM(d,p)$. See Section \ref{sec:RFM} for the formal definition. A related, but slightly different process known as the \emph{self-similar frog model} has been a useful tool for studying recurrence \cite{HJJ2, JJ3_log, HJJ_cover, HJJ_shape, josh_32}.

 Let $p'_d = \inf\{ p \colon \RFM(d,p) \text{ is recurrent}\}$ be the critical drift for the recursive frog model on $\mathbb T_d$. It follows from the dominance relation in \thref{lem:recRFM} that $p_d \leq p'_d$. As mentioned above, the usual frog model is difficult to couple on two trees of different degrees. Finding a coupling for the self-similar frog model also appears challenging. The discussion just after the definition of $\RFM(d,p)$ in Section \ref{sec:RFM} explains why in more detail. It is both useful for our main theorem and of independent interest that there is a coupling where $\RFM(d,p)$ is dominated by $\RFM(d+1,p)$. See \thref{lem:RFM'} for the proof. We use this to show that $p'_d$ is decreasing. 

 \begin{prop} \thlabel{prop:mono}
$p'_{d+1} \leq p'_d$ for all $d \geq 2$.
\end{prop}

The recursive frog model is useful because a coupling is possible across trees of different degrees. However, the coupling comes at the cost of removing a lot of awake frogs. Because so many frogs are removed, it is not obvious whether $\RFM(d,p)$ is ever recurrent. Old techniques do not apply easily here. We provide a new, simpler argument for recurrence for large enough $p$.
%\HOX{what does ``rely on our ability to choose the drift'' mean? -ST
%\color{blue}{I think that refers to the sentence right after Equation (6), which says that we can pick a drift strong enough to make the proof work. -EB}}
\begin{prop} \thlabel{prop:2}
$p'_2 \leq .4155$.
\end{prop}

 All previous results that establish recurrence for the frog model on trees rely on bootstrapping a recursive distributional equation involving $V(d,1/(d+1))$. See \eqref{eq:rde} for the equation. The recursive frog model is intuitively less recurrent than the self-similar frog model because more frogs are being removed. So, it is not clear that the bootstrapping approach will work. Fortunately, we find a simpler way to proceed. It starts with the usual recursive distributional equation, but uses the second moment method to finish. This is similar to an argument used to prove that the \emph{parking process} visits the origin infinitely often \cite{parking}. To finish we prove a 0-1 law for the recursive frog model. This is necessary because the recursive frog model is not covered by the 0-1 law in \cite{kosygina01}.
 %Contained in the proof is a $0$-$1$ law for the recursive self-similar frog model. Frog model $0$-$1$ laws are also explored in \cite{kosygina01,HJJ1}.

 %The proof is much simpler than before, but requires that we have a stronger drift towards the root than the natural one induced by the tree structure.

%\subsection*{Notation}

%We let $\mathbb T_d$ denote the rooted $d$-ary tree. This is the tree in which every vertex has $d$ children. For each $v \in \mathbb T_d$ let $\hat v$ denote the parent of $v$.
%Let $\mathbb T_d(v)$ be the infinite subtree rooted at $v \in \mathbb T_d$. For example, $\mathbb T_d(\rootvertex) = \mathbb T_d$.
%We write $f(n) = O(g(n))$ if $f(n) \leq C g(n)$ for some constant $C$ and all $n$, and $f(n) = \Omega(g(n))$ if the opposite inequality holds.

\section{The recursive frog model} \label{sec:RFM}

%We introduce a variant of the self-similar frog model which we call the \emph{recursive frog model} that allows us to couple different degree trees with the same drift.

The \emph{recursive frog model} $\RFM(d,p)$ has awake frogs that move towards the root at each step with probability $\rho = p/(1-p)$ when $p < 1/2$, and with probability $1$ for $p \geq 1/2$. If a frog reaches the root, it is removed. Once a frog moves away from the root, it moves to a uniformly sampled child vertex (possibly the vertex from which it just came) and will thereafter continue to move away from the root to a uniformly sampled child. Frogs are removed if they move away from the root and land on an already-visited site. To make removing frogs well-defined, we at each step sample one frog from the set of awake frogs and have it perform one step. The order frogs move does not change the law, and since the set of awake frogs grows by at most one at each time step, it is easy to see that every awake frog will be sampled infinitely often. $\RFM(d,p)$ earns its name because these modifications allow us to embed a recursive structure. 

Due to these modifications, any frog path in $\RFM(d, p)$ contains two stages. After being woken up, a frog in $\RFM(d, p)$ first moves directly towards the root (stage 1). Then, the frog either hits the root and gets removed or starts taking steps away from the root and gets removed if it hits an already-visited site (stage 2). The number of steps in stage 1 might be zero, in which case upon waking up, the first step of the frog is to move away from the root. If the frog is killed at the root, then stage 2 has zero steps. The path an awake frog follows in $\RFM(d,p)$ comes from the downward-loop-erased random walk of the corresponding frog in $\FM(d,p)$. This is the subrange of a simple random walk that ignores any loops created by steps away from the root.

The self-similar frog model $\SFM(d,p)$ has the same killing rules as $\RFM(d,p)$. The difference is that awake frogs follow the loop-erased path from their corresponding frog in $\FM(d,p)$. This path ignores all loops, not just downward ones. Unlike  $\RFM(d,p)$,  the first step an awake frog in $\SFM(d,p)$ takes away from the root cannot be to the one from which it just came. So an awake frog cannot cause its own death. Intuitively, this makes it so that $\SFM(d,p)$ invokes the killing rule less frequently than $\RFM(d,p)$. The difficulty with coupling $\SFM(d,p)$ and $\SFM(d+1,p)$ is that, even with the same drift, loop-erased random walk behaves differently on $\mathbb T_d$ than on $\mathbb T_{d+1}$. This is because a random walk on $\mathbb T_d$ is more likely to create loops. We could not find a natural way to have the displacement of a frog in $\SFM(d,p)$ align with a frog in $\SFM(d+1,p)$. Consequently, it is hard to see how to couple the two. The advantage of the recursive frog model is that there is a canonical way to align displacements. This is the starting point for our coupling.

%\subsection{Relating the recursive frog model to $\FM(d,p)$}

We begin by deducing $\RFM(d, p)$ visits the root no more than $\FM(d, p)$.
\begin{lemma}\thlabel{lem:recRFM}
If $\RFM(d, p)$ is recurrent, then $\FM(d, p)$ is recurrent. Hence, $p_d \leq p'_d.$
\end{lemma}
\begin{proof}
%We think about the one-per-site frog model as a collection of wake-up times of all frogs and the random walk paths associated to these frogs.
Consider another modification of the frog model, denoted by $\FM'(d, p)$, in which we impose the rule that if a frog makes a loop $v_{0}\to v_{1} \to \cdots v_{n} \to v_{0}$ with $v_{0}\to v_{1}$ being a step moving away from the root, then the frog does not wake up any sleeping frogs at the sites $v_{1}, \ldots, v_{n}$. Call such ignored loops \emph{silent}. Because the wake-up time for each sleeping frog in $\FM'(d, p)$ is no earlier than the corresponding wake-up time in $\FM(d, p)$, $\FM'(d, p)$ is a subprocess of $\FM(d, p)$, i.e., the set of awake frogs in $\FM'(d, p)$ is a subset of that in $\FM(d, p)$. Furthermore, $\RFM(d, p)$ can be constructed from ${\FM}'(d, p)$ by removing the silent loops and killing frogs whenever they step away from the root and hit a site that has already been visited. It is a simple exercise to prove that the probability a $p$-biased random walk reaches its parent vertex is $\rho$. It follows that the number of visits to the root in $\RFM(d, p)$ is no more than that in ${\FM}'(d, p)$.
\end{proof}

A key monotonicity property in $\FM(d,p)$ is that introducing more killing results in fewer visits to every site. It is not quite as obvious that  $\RFM(d,p)$ also enjoys this property. This is because when we remove a frog $f$ before it is killed via the killing rule of $\RFM(d, p)$, it will visit fewer sites. As a result, $f$ generates fewer already-visited sites, so other frogs might take more steps and survive longer. %than they otherwise would have if $f$ was not removed early.
\thref{lem:mono} explains why $\RFM(d,p)$ does in fact have this monotonicity property. For the sake of precision, when we say that a frog is \emph{removed early} this means that the frog is killed before the killing rule in $\RFM(d,p)$ would have killed it (if ever). 
%In such a process, we choose the frog to move at each step by sampling uniformly from the set of awake frogs.

\begin{lemma} \thlabel{lem:mono}
Any modification of $\RFM(d,p)$ with frogs removed early will visit each site of $\mathbb T_d$ no more than the usual $\RFM(d,p)$.
\end{lemma}

\begin{proof}
%The idea of the proof is that i
It is equivalent in law to view $\RFM(d,p)$ as a collection of infinite stacks of i.i.d.\ instructions at each vertex. Frogs move according to the instructions. Once used, the instruction is deleted from the stack. When frogs are removed early, at any time step after that, fewer instructions are used from each stack than in $\RFM(d,p)$ without removal. That is, the vertices along the path of the removed-early frog have one extra instruction in their stack. Either these sites remain visited one fewer time, or another frog comes along and continues the path of the removed-early frog. This does not result in more visited sites, only sites visited later in time, or possibly never. Therefore, each site is visited no more often than in the usual $\RFM(d,p)$.
\end{proof}

\section{Proof of \thref{prop:mono}}

The result follows immediately from \thref{lem:mono} and \thref{lem:RFM'}. In the latter, we will couple $\RFM(d, p)$ with a frog process on $\mathbb T_{d+1}$ so that the recurrence of one implies that of the other. The process on $\mathbb T_{d+1}$ is a modified version of $\RFM(d+1,p)$ with the possibility that frogs may be removed early.   By \thref{lem:mono} it follows that $\RFM(d+1,p)$ is also recurrent.

\begin{lemma} \thlabel{lem:RFM'}
$\RFM(d, p)$ can be coupled with a modified version of $\RFM(d+1, p)$ with frogs removed early so that recurrence is equivalent in both models.
% There exists a modified version of $\RFM(d+1,p)$ with more frogs removed that can be coupled to $\RFM(d,p)$ to be recurrent whenever $\RFM(d,p)$ is.
\end{lemma}

\begin{proof}
%Suppose $\RFM(d, p)$ is recurrent for some $p$ fixed.
We call the modified process $\RFM'(d+1,p)$. Frogs in $\RFM'(d+1,p)$ will follow the movement of frogs in $\RFM(d, p)$. % They will do so in such a way that their paths coincide with those in $\RFM(d+1,p)$, but with the possibility that certain frogs are removed earlier.  
To do this, we couple each frog $f$ in $\RFM(d, p)$ and its sleeping site $v\in \mathbb T_{d}$ with a unique frog $f'$ in $\RFM'(d+1, p)$ and its sleeping site $v'\in \mathbb T_{d+1}$, using the anotation $f \sim f'$, $v \sim v'$ to denote such coupling. The coupling is constructed recursively.

At time 0, there is a sleeping frog at every vertex of $\mathbb T_{d}$ in $\RFM(d, p)$, and, respectively, $\mathbb T_{d+1}$ in $\RFM'(d+1, p)$. The frog at the root vertex in each model is awake. These two awake frogs are coupled, and so are the two root vertices.
At each step, $t=1,2,\ldots$, a frog is picked uniformly at random from the set of awake frogs in $\RFM(d, p)$ and performs one step of random walk according to the recursive frog model.

Suppose an awake frog $f$ at vertex $v$ is selected. We will assume for now and justify later by induction that $f$ is already coupled with a unique awake frog $f'$ in $\RFM'(d+1, p)$ at some vertex $v'$.
For any vertex $v \in \mathbb T_{d}$ or $\mathbb T_{d+1}$, let $S(v, t)$ denote the number of child vertices of $v$  at which there is a sleeping frog after time step $t$ has occurred. For a frog $f$ we  write $S(f, t)$ for $S(v(f, t), t)$, where $v(f, t)$ is the vertex that frog $f$ lands on at time $t$.  The crux of our induction is that
%We claim \ref{rule-child} is holds for all time steps provided that
\begin{align}
S(f, t) +1 &= S(f', t),\text{ for all $t\ge 0$ and all frogs } f \text{ and } f' \text{ with } f\sim f'.
\label{eq:sleep-ineq},
\end{align}
Taking for granted that every awake frog $f$ in $\RFM(d,p)$ is coupled with an awake frog $f'$ and that \eqref{eq:sleep-ineq} holds, we can now define how $f'$ follows $f$:

\begin{enumerate}[label = (\alph*)]
\item \label{rule-parent} If $f$ moves to the parent vertex of $v$, then $f'$ moves to the parent vertex of $v'$.
\item \label{rule-child} If $f$ moves to a child vertex $v_{0}$ of $v$ and wakes up a sleeping frog $g$ at $v_{0}$, then $f'$ moves to a child vertex $v_{0}'$ of $v'$ chosen uniformly from those child vertices with a sleeping frog. $f'$ wakes the sleeping frog (denoted by $g'$) at that site. We couple the frogs
$g$ and $g'$ and the vertices $v_{0}$ and $v_{0}'$.
\item \label{rule-remove} If $f$ is removed, then $f'$ is also removed.
\end{enumerate}

We claim that $f'$ in $\RFM'(d+1, p)$ is moving and waking frogs as it would in $\RFM(d+1,p)$, but with the possibility of being removed early. First of all, rule \ref{rule-parent} ensures that $f'$ moves towards the root with the same probability as an awake frog  would do in $\RFM(d+1,p)$. Secondly, whenever $f$ wakes up a sleeping frog, $f'$ will always wake up one because \eqref{eq:sleep-ineq} ensures $S(f', t) > S(f, t)$. However, in the unmodified process $\RFM(d+1, p)$, given that a frog $f'$ moves away from the root and has $S(f', t)$ sleeping frogs immediately beneath itself, it would wake up a sleeping frog with probability
$$p(f',t):=\f{S(f',t)}{d+1},$$
whereas in $\RFM'(d+1, p)$, since $f'$ follows $f$ in $\RFM(d,p)$, the corresponding conditional probability of waking up a sleeping frog would be 
\begin{align}
\f{S(f,t)}{d} = \f{S(f',t) -1}{d} = p(f',t) - \f{d+1 - S(f',t)}{d(d+1)}. \label{eq:remove}
\end{align}
Note that the above probability is smaller than or equal to $p(f', t)$ since \eqref{eq:sleep-ineq} implies $1 \leq S(f',t) \leq d+1$. Therefore, marginally, $\RFM'(d+1,p)$ is equivalent to a modified version of $\RFM(d+1,p)$ with the additional removal rule that each awake frog $f'$, upon moving away from the root and hitting an unvisited site, is killed with probability $$0\leq \f{d+1 - S(f',t)}{ d (d+1)}\leq \f {1}{d+1}.$$ %This means the coupled process is a version of $\RFM(d+1,p)$ with some frogs removed early according to \eqref{eq:remove}.

Moreover, so long as \eqref{eq:sleep-ineq} holds, the coupled frogs in $\RFM(d,p)$ and $\RFM'(d+1,p)$ are always in bijection and at the same distance from the root. To see why, notice that whenever a frog wakes up in $\RFM(d, p)$, a sleeping frog must also wake up in $\RFM'(d+1, p)$. These two new frogs become coupled, which, together with the coupling rule \ref{rule-remove},  ensures that there exists a bijection between the set of the awake frogs (and their displacements from the root) in $\RFM(d,p)$ and the set of awake frogs (and their displacements from the root) in $\RFM'(d+1, p)$ at all times.  Therefore, the lemma will be established once we prove \eqref{eq:sleep-ineq}. We mark this proof as complete, and devote the rest of this section to that task.
\end{proof}

% where $S(f, t)$ denotes the number of sleeping frogs immediately beneath the frog $f$ at (the end of) step $t$. %, and $\sim$ represent the coupling relation of two vertices (or frogs).

%Moreover, two coupled frogs always move in the same direction, either toward the root or away from the root.  Since we start from two initially awakened frogs at the root, each subsequent pair of newly-awakened (and thus newly-coupled) frogs consist of two frogs at the same distance from the root.
% Denote by $v(f, t)$ the vertex that frog $f$ lands on at time $t$.

A key observation to establishing \eqref{eq:sleep-ineq} is that coupled frogs stay on coupled vertices until they are removed, i.e.,
\begin{lemma} \label{lem:frog-vertex-couple}
If $f\sim f'$, then $v(f, t) \sim v(f', t)$ for all $t \ge 0$.
\end{lemma}

\begin{proof}%[Proof of Lemma  \ref{lem:frog-vertex-couple}]
Recall that in $\RFM(d, p)$ an awake frog starts by taking steps toward the root (stage 1). After that, the frog is either removed due to hitting the root or begins moving away from the root (stage 2). In stage 2, when a frog moves away from the root, it either wakes up a sleeping frog or hits an already-visited site and gets removed. This means that in the coupling, while two coupled frogs walk away from the root, they couple newly-discovered vertices and frogs on the way (by waking up sleeping frogs) until they are removed. Thus two coupled frogs stay on coupled vertices when they take steps away from the root.

The steps moving toward the root can be taken care of by induction. Suppose the statement holds for all coupled frogs up to time $T$. At the $(T+1)$th step, we consider any two coupled frogs $f$ in $\RFM(d, p)$ and $f'$ in $\RFM'(d+1,p)$, which were initially woken up at some step $t_{0} \le T$ at vertices $v_{0}$ and $v_{0}'$ by two coupled frogs $g$ and $g'$, respectively. Upon waking up, vertices $v_{0}$ and $v_{0}'$ were immediately coupled. Moreover, since $g$ and $g'$ were at the parent vertices (denoted by $v_{1}$ and $v_{1}'$) of $v_{0}$ and $v_{0}'$ at $t_0-1$, then by the inductive hypothesis, $v_{1}$ and $v_{1}'$ must be coupled. Again, there must be two frogs $f_{1}$ and $f_{1}'$ initially sleeping at vertices $v_{1}$ and $v_{1}'$, because otherwise $v_{1}$ and $v_{1}'$ can not be coupled. Now consider the two frogs $f_{1}$ and $f_{1}'$ that are originally woken up at $v_{1}$ and $v_{1}'$ and repeat the argument above. It follows that the parent vertices of $v_{1}$ and $v_{1}'$ must also be coupled. Iterating this argument, we see that the unique path $v_{0}\to v_{1}\to \cdots\to \rootvertex$ from $v_{0}$ to the root in $\RFM(d, p)$ consists of vertices that are coupled with the vertices on the unique path $v_{0}'\to v_{1}'\to \cdots\to \rootvertex'$ from $v_{0}'$ to the root in $\RFM'(d+1, p)$, i.e.,
\[
v_{0}\sim v_{0}',\ v_{1}\sim v_{1}',\ \ldots, \ \rootvertex \sim \rootvertex'
\]
If at the $(T+1)$-th step, $f$ and $f'$ both move toward the root, then $f$ must still be in stage 1. Thus both frogs must move to some coupled vertices $v_{k}$ and $v_{k}'$.
\end{proof}

\begin{lemma}
The relation \eqref{eq:sleep-ineq} holds during the coupling described in \thref{lem:RFM'} .
\end{lemma}

\begin{proof}

With Lemma \ref{lem:frog-vertex-couple}, it is easy to see that
\begin{align}
\label{eq:vertex-ineq}
S(v, t) + 1 = S(v', t) \text{ for all }t\ge 0, v\in \mathbb T_{d}, v'\in\mathbb T_{d+1} \text{ and } v\sim v'
\end{align}
implies \eqref{eq:sleep-ineq}. We now prove that \eqref{eq:vertex-ineq} is preserved following the coupling rules by induction.

At time 0, only the two root vertices  $\rootvertex \in \RFM(d, p)$ and $\rootvertex' \in {\RFM}'(d+1, p)$ are coupled. We have
\[
S(\rootvertex, 0) =  d,\quad S(\rootvertex', 0) = d+1,
\]
and thus \eqref{eq:vertex-ineq} is satisfied. Now suppose \eqref{eq:vertex-ineq} is satisfied up to the $T$-th step for some $T>0$. By the discussion above, this means there is a bijection between awaked, and thus coupled, frogs after step $T$.

At the $(T+1)$th step, suppose an awake frog $f$ at vertex $v$ is picked in $\RFM(d, p)$, and denote its coupled frog in $\RFM'(d+1, p)$ sitting at vertex $v'$ by $f'$. We are done once we show that \eqref{eq:vertex-ineq} continues to hold after any of the three possible moves of $f$:
\begin{enumerate}[(i)]
\item $f$ moves away from the root to an already-visited site. In this case, $f$ is removed and so is $f'$. At the end of step $T+1$, the net change is to remove the pair of coupled and awake frogs $f$ and $f'$. The number of sleeping frogs beneath each vertex is unchanged. So \eqref{eq:vertex-ineq} continues to hold;
\item $f$ moves away from the root to $v_{0}$ with a sleeping frog $g$ and wakes it up. If this occurs, by the inductive hypothesis, it must be the case $S(v', T) = S(v, T)+1 \ge 2$;
According to the coupling rule \ref{rule-child}, $f'$ will wake up a sleeping frog $g'$ at a child vertex $v_{0}'$ of $v'$. We couple $v_{0}$ and $v_{0}'$.  By waking up $g$ and $g'$, we reduce the two numbers $S(v, T+1)$ and $S(v', T+1)$ by one simultaneously and \eqref{eq:vertex-ineq} continues to hold for $v$ and $v'$. In addition, the newly coupled vertices, $v_{0}$ and $v_{0}'$, have all their child vertices unvisited. Thus \eqref{eq:vertex-ineq} is preserved.
\item $f$ moves to the parent vertex of $v$. If this occurs, according to coupling rule \ref{rule-parent}, $f'$ moves to the parent vertex too ($f$ and $f'$ may be removed if $f$ hits the root). The number of sleeping frogs at each vertex does not change and neither do the collection of pairs of coupled vertices.
\end{enumerate}
Thus, \eqref{eq:vertex-ineq} is preserved after one time step, which completes the induction.
\end{proof}

\section{Proof of \thref{prop:2}} \label{sec:binary}

Another advantage of $\RFM(2,p)$ is that the number of visits to the root satisfies a recursive distributional equation. A similar but more complicated equation also holds for larger $d$. See Figure \ref{fig:ssf} for a visual representation of the following notation. Let $\rootvertex$ be the root of $\mathbb T_2$. The frog initially awake at the root will move to one of the two children of the root and then it, or the frog it wakes there, may move down another level.  Call these sites $\overline \rootvertex$ and $x$, respectively. Let $y$ be the sibling vertex of $x$.

Let $V_t$ be the number of visits to the root in $\RFM(2,p)$ with frogs placed at all sites up to distance $t$ from the root and the rest of the sites empty. By similar reasoning as \thref{lem:mono} we have $V_t \preceq V_{t+1}$ (in the usual sense of stochastic dominance) and thus there is a distributional limit $V:=V_\infty$. Let $V^x_t$ and $V^y_t$ be the number of visits to $\overline  \rootvertex $ from $x,y$, given that $x,y$, respectively, are visited. Let $A_t$ be the event that a frog ever enters the subtree rooted at $y$. Although the law for paths in $\RFM(2,p)$ is different than in the self-similar frog model from \cite{HJJ1}, it enjoys the same recursive properties. This is because both processes ($\RFM(d,p)$ and the self-similar frog model) have frogs follow non-backtracking paths and get removed when moving away from the root to already-visited sites.

Using arguments similar to \cite[Proposition 6, Proposition 7]{HJJ1} we have pairwise independence between $\ind{A_t}$ and $V_t^y$, and also that $V_t^y$ and $V_t^x$ are i.i.d.\ with distribution $V_{t-1}$. The law for random walk paths gives a slightly different recursion than at \cite[(2)]{HJJ1}. We have
\begin{align}V_{t} = \Bin(V_t^x +1, \rho) + \ind{A_t}\Bin( V_t^y, \rho). \label{eq:rde} \end{align}
The `$+1$' term in the first binomial variable comes from the frog initially sleeping at $\overline \rootvertex$. Analyzing the first and second moments of this recursive distributional equation is enough to deduce $V$ is infinite for $\rho$ large enough.

    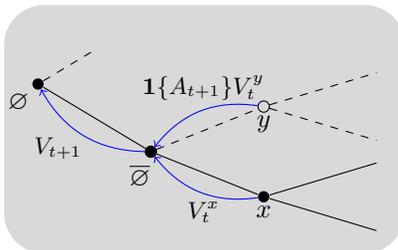
\begin{figure}[ht]
    \begin{center}
      %\tikzsetnextfilename{selfsimilarvertexnames}
\begin{tikzpicture}[scale = 1.5,tv/.style={circle,fill,inner sep=0,
    minimum size=0.15cm,draw}, starttv/.style={circle,fill,inner sep=0,
    minimum size=0.15cm},walklines/.style={very thick, blue,bend left=15}]
\fill[black!15!white,rounded corners=15pt] (-0.3, -1.5) rectangle (3.25, .7);
%\node at (1.5,.75) {$V_{t+1} = \Bin(V_t^x +1, \rho) + \ind{A_t}\Bin(V_t^y, \rho)$};
    \path (0,0) node[starttv] (R) {}
        (1,-0.6) node[tv] (L1) {}
        (1, 0.6)  node (L2) {}
        (2, -0.2) node[circle, draw, inner sep = 0, minimum size = .15cm] (L12) {}
        (2,-1) node[circle, fill, inner sep = 0, minimum size = .15cm] (L11) {}
        (L12)+(1,-0.3) coordinate (L121)
             +(1, 0.3) coordinate (L122)
        (L11)+(1,-0.3) coordinate (L111)
             +(1, 0.3) coordinate (L112);
  \draw
       (R)--(L1);
 \path[every node/.style={font=\sffamily\small}]
       (L1) edge [->,bend left,blue] node[below left = -.1 cm] {\;\;$\textcolor{black}{V_{t+1}}$} (R)
       (L11) edge [->,bend left,blue] node[below] {\;$\textcolor{black}{V_t^x}$} (L1)
       (L12) edge [->,bend right,blue] node[above = .03 cm] {\;$\textcolor{black}{\ind{A_{t+1}}V_t^y}$} (L1);
\node[below left] at (R){$\emptyset$};
\node[below] at (L12){$y$};
\node[below] at (L11){$x$};
\node[below =.05 cm] at (L1){$\overline \emptyset$\;\;\;};
    \draw (L1)--(L11);
    \draw[dashed] (R) -- (.5,.3);
%    \node[align=left] at (1.3,.5) { \tiny{frogs cannot} };
 %   \node[align=left] at (1.11,.3){ \tiny{visit here}};
%    \node at (1.3,.5) { \tiny{frogs cannot} };
%    \node at (1.11,.3){ \tiny{visit here}};
       \draw[dashed] (L1) -- (L12);
  \draw[dashed]
       (L12)--(L122) (L12)--(L121);

  \draw     (L11)--(L112) (L11)--(L111);

\end{tikzpicture}
    \end{center}
    \caption{$V_{t+1}$ is the total number of visits to $\emptyset$ in $\RFM(d,p)$ with sleeping frogs placed to distance $t+1$ from the root. It can be expressed as a binomial thinning of the number of visits to $\overline \rootvertex$. These quantities are i.i.d.\ and distributed like $V_t$.
    %It can be decomposed in terms of the visits to $\emptyset'$. The quantities $V_t^x$ and $\ind{A_t}V_t^y$ are the number of visits to $\emptyset'$ from frogs originally in $\mathbb T_2(y)$ and $\mathbb T_2(x)$. $A_t$ is the event that $y$ is visited. We have $V_t^x$ and $V_t^y$ are i.i.d.\ with distribution $V_t$. However, there is some dependence between $A_t$ and $V_t^x+1$.
    }
    \label{fig:ssf}
  \end{figure}

\begin{proof}[Proof of \thref{prop:2}]
Let $x_t = E V_t^2 / (EV_t)^2$. We will prove that $\sup_t x_t = C < \infty$. It follows from the Paley-Zygmund inequality that
\begin{align}
P(V_t > E V_t /2 ) \geq (4 x_t)^{-1} \geq (4C)^{-1} \qquad \text{ for all $t \geq 1$}.\label{eq:V}
\end{align}
We will also show that $EV_t \to \infty$ as $t \to \infty$, and the above line implies $P(V = \infty) > 0$.
%It then follows from the 0-1 law in the appendix that $P(V_{\infty} = \infty) =1$.

Taking expectation in \eqref{eq:rde} and using independence between $\ind{A_{t+1}}$ and $V_{t+1}^y$ gives
\begin{align}E V_{t+1} =  \rho(1 + P(A_{t+1}) )E V_t + \rho. \label{eq:1st}
\end{align}
It is easy to show that $P(A_{t+1}) \geq 1- \prod_{i=0}^{t} (1- \rho^i (1- \rho)/2).$ This is because there always exists a line segment $L_{t+1}$ from $\overline  \emptyset$ to a vertex at distance $t+1$ from the root along which all the frogs have been woken up. For the frog on this ray at distance $i+1$ from the root to visit $y$ it must take $i$ steps toward the root, then move to $y$. This occurs with probability $\rho^i(1-\rho)/2$. Since there are $t$ frogs along this line segment, we obtain the claimed bound on $P(A_{t+1})$ by only considering these frogs guaranteed to be awake.

A computer can easily verify that $\rho(1+P(A_{51})) >1$ for $\rho >.7107.$ Converting from $\rho = p/(1-p)$ back to $p$ implies that   this holds for $p>.4155.$ Using \eqref{eq:1st}, for such $p$ we have $\epsilon >0$ so that $E V_{t+1} \geq (1+ \epsilon) E V_t + \rho$ for $t\geq 50$. It follows that $E V_t$ diverges as $t \to \infty$. This alone is not enough to conclude that the root is visited infinitely often almost surely. To establish this, we need to control the second moment.

Let $X_t = \Bin(V_t^x + 1, \rho)$ and $Y_t= \Bin(V_t^y, \rho)$ so that $V_{t+1} = X_t + \ind{A_{t+1}}Y_t$ and thus
\begin{align}
V_{t+1}^2 = X_t^2 + \ind{A_{t+1}}Y_t^2 + 2 \ind{A_{t+1}} X_t Y_t.
\end{align}
Taking expectations and using independence as well as the bound $\ind{A_t}\leq 1$ we have
\begin{align}
E V_{t+1}^2 \leq E X_t^2 + P(A_{t+1}) E Y^2_t +  2E X_t E Y_t.\label{eq:square}
\end{align}
Using the formula for the second moment of a random sum of i.i.d.\  $Z_i\overset{d} = \Ber(\rho)$
\begin{align*}E \left( \sum_{i=1}^{N} Z_i \right)^2 =  \rho(1-\rho)E N  + \rho^2 EN^2,
\end{align*}
we have $E X_t^2 = \rho(1-\rho) E (V_t+1) + \rho^2 E (V_t+1)^2,$ and similarly $E Y_t^2=\rho(1-\rho) E V_t + \rho^2 E V_t^2$. Plugging these expressions into \eqref{eq:square} and gathering smaller order terms yields
\begin{align}
E V^2_{t+1} \leq \rho^2(1+ P(A_{t+1})) E V_t^2 + 2\rho^2 (EV_t)^2 + O(E V_t). \label{eq:2nd}
\end{align}
Squaring \eqref{eq:1st} and ignoring smaller order terms gives
\begin{align}(E V_{t+1})^2 \geq (\rho(1+P(A_{t+1})) EV_t)^2.\label{eq:2}\end{align}
Recall that $x_t =  EV_t^2 / (E V_t)^2$. Dividing \eqref{eq:2nd} by \eqref{eq:2} gives
$$x_{t+1} \leq \f{1}{1+ P(A_{t+1})} x_t + O(1).$$
Since $1+P(A_{t+1}) \geq 1+\epsilon  >1$ for all $t$, the leading coefficient is less than $1$. This ensures that $\sup_t x_t = C < \infty$ which gives \eqref{eq:V}.

There is a quick way to go from $P(V=\infty)>0$  to $P(V=\infty) = 1$. Recall the definition of $L_{t+1}$ from just below \eqref{eq:1st}. We can extend this to obtain a ray $L$ from the root to $\infty$ with an awake frog at each site. Let $\rootvertex_t$ be the site at distance $t$ on this ray. The awake frog at $\rootvertex_t$ moves to the child $y_t \notin L$ beneath it with probability $(1- \rho)/2$. When this occurs, an independent $V$-distributed number of frogs will visit $\rootvertex_t$. If this quantity is infinite, then $\rootvertex$ is visited infinitely often. Since $P(V=\infty) >0$ and there are infinitely many independent trials along $L$, we must have $V$ is infinite almost surely. 
%
%We can apply the general 0-1 law for frog models from \cite{kosygina01} and deduce that $\FM(d,p)$ satisfies a 0-1 law so must be recurrent.
\end{proof}

\begin{comment}
\section{A lower bound via BRW}

\begin{lemma}\thlabel{lem:brw}
	$(2 - \sqrt 2)/4 = q^* \leq p_d$ for all $d\geq 2$.
\end{lemma}

\begin{proof}
Consider the weight function $$W_n = \sum_{f\in A_n} e^{-\theta \abs{f}},$$
  with $A_n$ the set of frogs awake at time~$n$, and for $f\in A_n$,  $\abs{f}$
  denotes the level of $f$ on the tree (that is, its distance from the root). The parameter $\theta$ is chosen to minimize the maximum expected contribution of a single frog to $W$:
  $$m = pe^{\theta} + 2(1-p)e^{-\theta}.$$
This corresponds to the BRW that branches every step away from the root. Using similar reasoning as the conclusion of \cite[Proposition 15]{HJJ2}, we have $W_n \to 0$ whenever $m(\theta) <1$. Such convergence implies that $\FM(d,p)$ is transient since $W_n \geq 1$ at each visit to the root.
The minimum of $m$ for $\FM(d,p)$ occurs at $\theta = \log(2 (p^{-1} -1))/2.$ Using this value gives $m =  2 \sqrt{p^{-1} -1 }$, which is less than $1$ whenever $p< q^*.$
\end{proof}
\end{comment}

\section{Proof of \thref{prop:coupling}}

\begin{figure}[ht]
	\includegraphics[width = 11 cm]{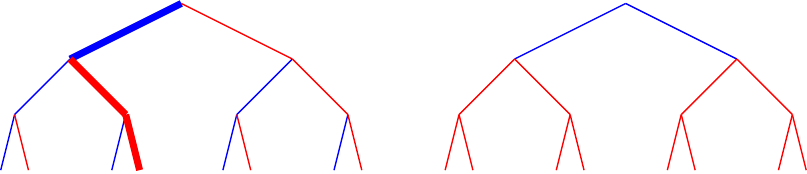}
	\caption{Each segment $L_v$ is associated to a $k$-ary tree in $\mathbb T_{kd}$. The bolded line $L_{122} \subseteq \mathbb T_2$ on the left corresponds to the binary tree in $\mathbb T_4$ shown on the right. The color coding represents the coordinate entries of $v$.} \label{fig:embedding}
\end{figure}

\begin{proof}
We impose coordinates on $\mathbb T_d$ by writing a vertex $v$ at distance $n$ from the root as $v = x_1\cdots x_n$ with $x_i \in \{1,2,\hdots, d\}$. For $1\leq i \leq d$ define the set-valued function $G(i) = \{k(i-1) +1,\hdots,ki\}$. Let $L_v$ be the set of vertices $\rootvertex = v_0,v_1,\hdots, v_n = v$ on the shortest path from the root to $v$. For each $L_v$ we define a subgraph of $\mathbb T_{k}$:
$$\mathbb T_{k}(L_v) = \bigcup_{x_1\cdots x_i \in L_v} G(x_1)\times \cdots \times G(x_i).$$
If $v$ has distance $n$ from the root, then $\mathbb T_{k}(L_v)$ is a $k$-ary tree of height $n$. Call vertices $\{ v' \in \mathbb T_{k} (L_v) \colon |v'| = n\}$ the \emph{leaves}. The embedding is such that for $v, v' \in \mathbb T_d$ we have
\begin{align}
\mathbb T_{k}( L_v)\cap \mathbb T_{k}(L_{v'}) =\mathbb T_{k}(L_v \cap L_{v'}). \label{eq:intersec}	
\end{align}

We will define a modified version of $\FM(kd,p)$ that sometimes removes frogs. Awake frogs $f'$ in the modified version will be coupled to a unique frog $f$ in $\FM(d,p)$. The rules for the coupling are that:
\begin{enumerate}[label=(\roman*)]
	\item If $f$ moves towards the root, then $f'$ moves towards the root.
	\item Suppose that $f$ is at $x_1\cdots x_n$ and $f'$ is at $x_1'\cdots x_n'$. If $f$ moves away from the root to $x_1\cdots x_n x_{n+1}$, then $f'$ moves to a uniformly random vertex in $x_1'\cdots x_n' \times G(x_{n+1})$.
	\item $f'$ only wakes a frog when $f$ does. Upon doing so these newly awakened frogs are also coupled.
\end{enumerate}
These rules ensure that $f$ and $f'$ have the same displacement from the root so Rule (ii) always holds. Moreover, Rule (ii) combined with \eqref{eq:intersec} ensure that the first visit to $v \in \mathbb T_d$ corresponds to the first visit to a leaf of $\mathbb T_k(L_v)$. So, when the frog at $v$ is woken by $f$, then there will be a sleeping frog at whatever leaf vertex of $\mathbb T_k(L_v)$ that $f'$ moves to. It follows that Rule (iii) holds for all steps in the coupling. The resulting process is a restricted version of $\FM(kd,p)$ that dominates $\FM(d,p)$. This gives the claimed result.
%
%These rules ensure that $f$ and $f'$ have the same displacement from the root. Moreover, Rule (ii) combined with \eqref{eq:intersec} ensure that the first visit to $v \in \mathbb T_d$ corresponds to the first visit to a leaf of $\mathbb T_k(L_v)$. So, when the frog at $v$ is woken by $f$, then there will be a sleeping frog at whatever leaf vertex of $\mathbb T_k(L_v)$ the frog $f'$ moves to. So whenever $f$ wakes a frog, so does $f'$. The resulting process is a restricted version of $\FM(kd,p)$ that dominates $\FM(d,p)$. This gives the claimed result.
%>>>>>>> 209f368f751e9528cb0bc4527be904846e140ee1
\end{proof}

\begin{remark}
	This coupling is wasteful; only one of the $k^n$ frogs at the leaves of $\mathbb T_k(L_v)$ can ever be woken. Nonetheless, it is the only  one we could find with $\FM(d,p) \preceq \FM(d',p)$ for some $d< d'$. 
\end{remark}

\section*{Acknowledgements}
\noindent We thank Mina Ossiander  for raising this question at the Oregon State University Probability Seminar in 2015. 
%Rick Durrett also provided helpful feedback throughout the process of solving and writing up this problem.
 Yufeng Jiang was partially supported by the 2018 Duke Opportunities in Math program at Duke University.

\bibliographystyle{amsalpha}
\bibliography{frog_paper_cover.bib}

\newcommand{\etalchar}[1]{$^{#1}$}
\providecommand{\bysame}{\leavevmode\hbox to3em{\hrulefill}\thinspace}
\providecommand{\MR}{\relax\ifhmode\unskip\space\fi MR }
\providecommand{\MRhref}[2]{%
  \href{http://www.ams.org/mathscinet-getitem?mr=#1}{#2}
}
\providecommand{\href}[2]{#2}
\begin{thebibliography}{DGH{\etalchar{+}}17}

\bibitem[DGH{\etalchar{+}}17]{nina_drift2}
Christian D{\"o}bler, Nina Gantert, Thomas H{\"o}felsauer, Serguei Popov, and
  Felizitas Weidner, \emph{{Recurrence and Transience of Frogs with Drift on
  $\mathbb{Z}^d$}}, available at arXiv:1709.00038, 2017.

\bibitem[DGJ{\etalchar{+}}17]{parking}
Micheal Damron, Janko Gravner, Matthew Junge, Hanbaek Lyu, and David Sivakoff,
  \emph{Parking on transitive unimodular graphs}, arXiv preprint
  arXiv:1710.10529 (2017).

\bibitem[DP14]{dobler_drift}
Christian Döbler and Lorenz Pfeifroth, \emph{Recurrence for the frog model
  with drift on ℤd}.

\bibitem[FMS04]{mono}
L.~R. Fontes, F.~P. Machado, and A.~Sarkar, \emph{The critical probability for
  the frog model is not a monotonic function of the graph}, Journal of Applied
  Probability \textbf{41} (2004), no.~1, 292--298.

\bibitem[GNR17]{ghosh_drift}
Arka Ghosh, Steven Noren, and Alexander Roitershtein, \emph{On the range of the
  transient frog model on {$\mathbb{Z}$}}, Adv. in Appl. Probab. \textbf{49}
  (2017), no.~2, 327--343. \MR{3668379}

\bibitem[GS09]{nina_drift1}
Nina Gantert and Philipp Schmidt, \emph{Recurrence for the frog model with
  drift on {$\mathbb Z$}}, Markov Process. Related Fields \textbf{15} (2009),
  no.~1, 51--58. \MR{2509423 (2010g:60170)}

\bibitem[HJJ16]{HJJ2}
Christopher Hoffman, Tobias Johnson, and Matthew Junge, \emph{From transience
  to recurrence with {P}oisson tree frogs}, Ann.\ Appl.\ Probab. \textbf{26}
  (2016), no.~3, 1620--1635. \MR{3513600}

\bibitem[HJJ17a]{HJJ_shape}
\bysame, \emph{Infection spread for the frog model on trees}, available at
  arXiv:1710.05884, 2017.

\bibitem[HJJ17b]{HJJ1}
\bysame, \emph{Recurrence and transience for the frog model on trees}, Ann.
  Probab. \textbf{45} (2017), no.~5, 2826--2854. \MR{3706732}

\bibitem[HJJ18]{HJJ_cover}
\bysame, \emph{Cover time for the frog model on trees}, available at
  arXiv:1802.03428, 2018.

\bibitem[JJ16a]{JJ3_log}
Tobias Johnson and Matthew Junge, \emph{The critical density for the frog model
  is the degree of the tree}, Electron.\ Commun.\ Probab. \textbf{21} (2016),
  Paper No. 82, 12. \MR{3580451}

\bibitem[JJ16b]{JJ3_order}
\bysame, \emph{Stochastic orders and the frog model}, to appear in
  \emph{Annales de l'Institut Henri Poincar\'e}, available at arXiv:1602.04411,
  2016.

\bibitem[KZ17]{kosygina01}
Elena Kosygina and Martin P.~W. Zerner, \emph{A zero-one law for recurrence and
  transience of frog processes}, Probability Theory and Related Fields
  \textbf{168} (2017), no.~1, 317--346.

\bibitem[LMP05]{improved}
{\'E}lcio Lebensztayn, F{\'a}bio~P Machado, and Serguei Popov, \emph{An
  improved upper bound for the critical probability of the frog model on
  homogeneous trees}, Journal of statistical physics \textbf{119} (2005),
  no.~1-2, 331--345.

\bibitem[Ros17a]{josh_drift}
Josh Rosenberg, \emph{The frog model with drift on {$\mathbb R$}}, Electron.
  Commun. Probab. \textbf{22} (2017), Paper No. 30, 14. \MR{3663101}

\bibitem[Ros17b]{josh_drift2}
\bysame, \emph{The nonhomogeneous frog model on $\mathbb{Z}$}, available at
  arXiv:1707.07749, 2017.

\bibitem[Ros17c]{josh_32}
\bysame, \emph{Recurrence of the frog model on the 3,2-alternating tree},
  available at arXiv:1701.02813, 2017.

\end{thebibliography}

\end{document}